\documentclass{amsart}
\usepackage{amsthm,amsmath,amsfonts}
\usepackage{mathrsfs}
\usepackage[colorlinks=false,linktocpage=true]{hyperref}
\usepackage{tocvsec2}
\usepackage{latexsym}
\begin{document}
\title[SDDEs limits solutions]{SDDEs limits solutions to sublinear reaction-diffusion SPDEs}
\author{Hassan Allouba}
\address{Department of Mathematical Sciences,
Kent State University, Kent, Ohio 44242, USA}
\email{allouba@mcs.kent.edu}

\date{October 1, 02}
\thanks{Partially supported by grant MDA904-02-1-0083 from NSA.}

\begin{abstract}
  We start by introducing a new definition of solutions to heat-based
  SPDEs driven by space-time white noise: SDDEs
  (stochastic differential-difference equations)  limits solutions.
  In contrast to the standard direct definition of   SPDEs solutions;
  this new notion, which builds on and refines our SDDEs approach to
  SPDEs from earlier work, is entirely based on the approximating SDDEs.
  It is applicable to, and gives a multiscale view of, a variety of SPDEs.
  We extend this approach in related work to other heat-based SPDEs
  (Burgers, Allen-Cahn, and others) and to the difficult case of SPDEs with
  multi-dimensional spacial variable.  We focus here on
  one-spacial-dimensional reaction-diffusion SPDEs; and we prove the
  existence of a SDDEs limit solution to these equations under
  less-than-Lipschitz conditions on the drift and the diffusion coefficients,
  thus extending our earlier SDDEs work to the nonzero drift case.
  The regularity of this solution is obtained  as a by-product
  of the existence estimates.  The uniqueness in law of our SPDEs follows,
  for a large class of such drifts/diffusions, as a simple extension of our
  recent Allen-Cahn uniqueness result.  We also examine briefly,
  through order parameters $\epsilon_1$ and $\epsilon_2$ multiplied by the
  Laplacian and the noise, the effect of letting
  $\epsilon_1,\epsilon_2\to 0$ at different speeds.  More precisely,
  it is shown that the ratio $\epsilon_2/\epsilon_1^{1/4}$ determines
  the behavior as $\epsilon_1,\epsilon_2\to 0$.
\end{abstract}

\maketitle
\numberwithin{equation}{section}
\newtheorem{thm}{Theorem}[section]
\newtheorem{cor}[thm]{Corollary}
\newtheorem{lem}[thm]{Lemma}
\newtheorem{prop}[thm]{Proposition}
\newtheorem{defn}[thm]{Definition}
\newtheorem{rem}[thm]{Remark}
\allowdisplaybreaks
\font\tenscrpt=eusm10
\font\sevenscrpt=eusm10 scaled 700
\font\fivescrpt=eusm10 scaled 500
\newfam\eusmfam
\textfont\eusmfam=\tenscrpt
\scriptfont\eusmfam=\sevenscrpt
\scriptscriptfont\eusmfam=\fivescrpt
\def\eus#1{{\fam\eusmfam\relax#1}}

\section{Introduction and statements of results}

We consider the parametrized space-time white noise driven SPDE on
$\eus{R}_{T}\overset{\triangle}{=} \mathbb{T}\times\mathbb{R}
=[0,T]\times\mathbb{R}$:
\begin{equation}
\begin{gathered}
\frac{\partial U_{\epsilon_1,\epsilon_2}}{\partial t}
=\frac{\epsilon_1}2\Delta U_{\epsilon_1,\epsilon_2}+b(U_{\epsilon_1,\epsilon_2})
+\epsilon_2a(U_{\epsilon_1,\epsilon_2}) \frac{\partial^2 W}{\partial t\partial x};
\quad (t,x)\in(0,T]\times\mathbb{R}, \\
 U_{\epsilon_1,\epsilon_2}(0,x)=\xi(x) ; \quad x\in \mathbb{R},
\end{gathered} \label{G}
\end{equation}
where $T>0$ is fixed but arbitrary, $\Delta$ is the Laplace operator
in space, and $a,b:\mathbb{R}\to\mathbb{R}$ are Borel
measurable.  $W(t,x)$ is the Brownian sheet corresponding to the driving
space-time white noise--with intensity Lebesgue measure--written formally as
$\partial^2W/\partial t\partial x$.  As in Walsh \cite{WA},
white noise is regarded as a continuous orthogonal martingale
measure, which we denote by $\eus{W}$, with the corresponding Brownian sheet
as the random field induced by $\eus{W}$ in the usual way.  $\xi(x)$ is
taken to be a continuous bounded deterministic function.
The parameters $\epsilon_1,\epsilon_2>0$ are order parameters, which allow
us to control the competing effects of the Laplacian
$\Delta$ and the driving space-time white noise
$\partial^2W/\partial t\partial x$.  We denote the SPDE in \eqref{G} by
$e^{\epsilon_1,\epsilon_2}_{\mbox{\tiny heat}} (a,b,\xi)$

Before going into the statements of our results, let's highlight one of the main features of this article.
We introduce and formalize the notion of SDDEs limit solutions (see Definition \ref{limitsolns} and Remark \ref{defnrem} below), and most of our treatment here focuses on this class of solutions to
$e^{\epsilon_1,\epsilon_2}_{\mbox{\tiny heat}} (a,b,\xi)$ whose elements are limits of an approximating sequence of stochastic-differential-difference
equations (SDDEs) (see \cite{ASDDE1, AD}).  SDDEs are obtained from $e^{\epsilon_1,\epsilon_2}_{\mbox{\tiny heat}} (a,b,\xi)$ by discretizing space but leaving time continuous.
We have used SDDEs before in \cite{AD,ASDDE1} to give a non-nonstandard proof of Reimers' existence result for $e_{\mbox{\tiny heat}}(a,0,\xi)$
(the case of zero drift) when $a$ is continuous and satisfies a linear growth condition.  In addition to extending our existence
proof there to the case of nonzero continuous drift (Theorem \ref{ww}) and examining the
effects of the order parameters $\epsilon_1,\epsilon_2$ on $e^{\epsilon_1,\epsilon_2}_{\mbox{\tiny heat}} (a,b,\xi)$ (Theorem \ref{thm2}), our new definition (Definition \ref{limitsolns}) of solutions to
SPDEs as limits of approximating SDDEs
establishes a general approach of SPDEs in which solutions to the SPDE in question are defined {\it entirely }in terms
of its approximating SDDEs and their limits (more on this approach and its implications in $d$-dimensional space as well as for other
SPDEs is detailed in \cite{Adim}).
It is important to note that (even in one dimension spacial variable) this is different from, and has several
advantages not shared with, the traditional direct approach (in which solutions to $e^{\epsilon_1,\epsilon_2}_{\mbox{\tiny heat}} (a,b,\xi)$ are understood in the sense of either
the standard test function \eqref{TFF} or Green function formulations \eqref{GFF} of  $e^{\epsilon_1,\epsilon_2}_{\mbox{\tiny heat}} (a,b,\xi)$).
In addition to the obvious numerical advantage;  1.~this is a multiscale approach which
allows us to view the model under consideration in two different scales simultaneously (the microscopic-in-space SDDE scale and their limiting SPDE scale),
and thus be able to see which properties of the SPDEs being approximated is captured by their SDDEs and which ones are different (e.g.~Proposition \ref{Expbd},
Lemma \ref{boundsonsdifferences}, and Lemma \ref{boundsontdifferences} show how some regularity properties for $e^{\epsilon_1,\epsilon_2}_{\mbox{\tiny heat}} (a,b,\xi)$ are
captured by their SDDEs, and the proof of Theorem \ref{thm2} part (ii) is the same for both $e^{\epsilon_1,\epsilon_2}_{\mbox{\tiny heat}} (a,b,\xi)$ and the corresponding SDDEs, while the proof
of Theorem \ref{thm2} part (i) is an example of an argument which holds in the continuous setting but not in the SDDE one, possibly pointing to different behaviors
in the two different scales),  2.~the role of the heat Green function for $e^{\epsilon_1,\epsilon_2}_{\mbox{\tiny heat}} (a,b,\xi)$ is played by a continuous-time random walk density in the case
of SDDEs, allowing us to use powerful and simpler random walk arguments like coupling (see Lemma \ref{5thinequality}) and others to get needed estimates to prove existence and
regularity for $e^{\epsilon_1,\epsilon_2}_{\mbox{\tiny heat}} (a,b,\xi)$ under mild conditions on $a$ and $b$ (more on the intimate connection between the Green function and the random walk density is also
detailed in \cite{Adim}), and 3.~unlike the usual Green's function formulation in the direct approach to SPDEs, these SDDEs
make sense as {\it real-valued} random fields in any spatial dimension $d$, and we use them in \cite{Adim} to extend our
Definition \ref{limitsolns} below to a class of solutions for SPDEs in any dimension.  In this last regard, SDDEs limit solutions to SPDEs are similar to
Lions-Crandal notion of viscosity solutions in that they are defined as limits of more ``regular'' solutions.

Consider the sequence of lattices ${(\mathbb{X}_n)\,}_{n=1}^{\infty}$ defined
by
\begin{equation*}
\mathbb{X}_n \overset{\triangle}{=}  \{\cdots, -2\delta_n, -\delta_n, 0, \delta_n, 2\delta_n, \cdots \},
\end{equation*}
where $\delta_n\to0$ as $n\to\infty$.  Then, following \cite{ASDDE1}, $e^{\epsilon_1,\epsilon_2}_{\mbox{\tiny heat}} (a,b,\xi)$
may be approximated by the sequence of SDDEs $\{e_{\mbox{\tiny heat}}^{\mbox{\tiny SDDE}}(a,b,\xi,n)\}_{n\in\mathbb{N}}$:
\begin{equation}
\begin{gathered}
\begin{aligned}
 d\tilde{U}^x_{n,\epsilon_1,\epsilon_2}(t)
 &= \big[\frac{\epsilon_1}{2}\Delta_n\tilde{U}^x_{n,\epsilon_1,\epsilon_2}(t)
 +b(\tilde{U}^x_{n,\epsilon_1,\epsilon_2}(t))\big]dt\\
 &\quad +\epsilon_2a(\tilde{U}^x_{n,\epsilon_1,\epsilon_2}(t))
 \dfrac{dW_n^x(t)}{\delta_n^{1/2}};
 \quad (t,x)\in(0,T]\times\mathbb{X}_n, \end{aligned}\\
\Tilde{U}_{n,\epsilon_1,\epsilon_2}^x(0) = \xi(x); \quad x\in\mathbb{X}_n,
\end{gathered} \label{SDDE}
\end{equation}
where $\Delta_n f(x)$ is the $n$-th discrete Laplacian
\begin{equation}
\Delta_nf(x) = \frac{f(x+\delta_n)-2f(x)+f(x-\delta_n)}{\delta_n^2}.
\label{LG}
\end{equation}
For each $n\in\mathbb{N}$, we think of $W_n^x(t)$ as a sequence of independent standard Brownian motions
indexed by the set $\mathbb{X}_n$ (independence within the same lattice).  We also
assume that if $m\neq n$ and $x\in\mathbb{X}_m\cap\mathbb{X}_n$ then $W_m^x(t)=W_n^x(t)$, and if $n>m$ and
$x\in\mathbb{X}_n\setminus\mathbb{X}_m$ then $W_m^x(t)=0$.

\begin{defn} \label{dfn1.1}
Fix $\epsilon_1,\epsilon_2$.  A solution to the SDDE system
$\{e_{\mbox{\tiny heat}}^{\mbox{\tiny SDDE}}(a,b,\xi,n)\}_{n=1}^\infty$;
 with respect to the
Brownian (in $t$) system $\{W_n^x(t)\}_{(n,x)\in\mathbb{N}\times\mathbb{X}_n}$
on the filtered probability space $(\Omega,\eus{F}\{\eus{F}_t\},\mathbb{P})$;
is a sequence of real-valued processes
$$
\left\{\tilde{U}_{n,\epsilon_1,\epsilon_2}
=\{\tilde{U}^x_{n,\epsilon_1,\epsilon_2}(t);(t,x)\in(\mathbb{T}\times\mathbb{X}_n)\}
\right\}_{n=1}^\infty
$$
with continuous sample paths in $t$ for each fixed $x\in\mathbb{X}_n$ and
$n\in\mathbb{N}$ such that,
for every $(n,x)\in\mathbb{N}\times\mathbb{X}_n$,
$\tilde{U}^x_{n,\epsilon_1,\epsilon_2}(t)$ is $\eus{F}_t$-adapted, and
\begin{equation}
\begin{aligned}
\tilde{U}^x_{n,\epsilon_1,\epsilon_2}(t)
&=\int_0^t\left[\frac{\epsilon_1}{2}\Delta_n\tilde{U}^{x}_{n,\epsilon_1,
\epsilon_2}(s)+b(\tilde{U}^{x}_{n,\epsilon_1,\epsilon_2}(s))\right]ds
+\epsilon_2a(\tilde{U}^{x}_{n,\epsilon_1,\epsilon_2}(s))
\frac{dW_n^x(s)}{\delta_n^{1/2}} \\
&\quad+\xi(x);\quad(t,x)\in\mathbb{T}\times\mathbb{X}_n,\ n\in\mathbb{N},\mbox{ a.s. }\mathbb{P}.
\end{aligned} \label{ISDDE}
\end{equation}
A solution is said to be strong if
$\{W_n^x(t)\}_{(n,x)\in\mathbb{N}\times\mathbb{X}_n}$ and
$(\Omega,\eus{F}\{\eus{F}_t\},\mathbb{P})$ are fixed a priori$;$
and with
\begin{equation}
\eus{F}_t=\sigma\left\{\sigma\left(W_n^x(s);0\le s\le t,x\in\mathbb{X}_n,n\in\mathbb{N}\right)\cup
\eus{N}\right\};\quad t\in\mathbb{T},
\label{filt}
\end{equation}
where $\eus{N}$ is the collection of null sets
$$
\Big\{O:\exists\,G\in\sigma\Big(
\bigcup_{t\ge0}\sigma\left(W_n^x(s);0\le s\le
t,x\in\mathbb{X}_n,n\in\mathbb{N}\right)\Big),\
O\subseteq G\ \mbox{and}\ \mathbb{P}(G)=0\Big\}
$$
A solution is termed weak if we are free to
choose $(\Omega,\eus{F}\{\eus{F}_t\},\mathbb{P})$ and the Brownian system on it and without requiring
$\eus{F}_t$ to
satisfy $\eqref{filt}$.
\end{defn}

Unless otherwise stated all filtrations are assumed to satisfy the usual conditions,
and any filtered probability space with such filtration is called a usual probability
space.

Now, as in Lemma 2.1 in \cite{ASDDE1},
we easily have the following representation and existence result for our approximating SDDEs

\begin{lem} \label{GSDDE}
Under the conditions
\begin{equation}
\begin{aligned}
&(a)\quad  a(u) \mbox{ and } b(u)\mbox{ are continuous in }u;
\quad
u\in\mathbb{R},\\
&(b)\quad  a^2(u)\le K(1+u^{2})\mbox{ and }b^2(u)\le
K(1+u^{2});\quad
u\in\mathbb{R}, \\
&(c)\quad \xi \mbox{ is continuous, nonrandom, and bounded on
}\mathbb{R},
\end{aligned} \label{acond}
\end{equation}
for some constant $K>0$, the SDDE system $\eqref{SDDE}$ is equivalent to the discrete-space
continuous-time Green function formulation
\begin{equation}
\begin{split}
\tilde{U}^x_{n,\epsilon_1,\epsilon_2}(t)
&= \sum_{y\in\mathbb{X}_n}\int_0^t {Q_{\delta_n,\epsilon_1}^{t-s;x-y}}
\Big[\epsilon_2{a(\tilde{U}_{n,\epsilon_1,\epsilon_2}^y(s))} \frac{dW_n^y(s)}
{\delta_n^{1/2}}+{b(\tilde{U}_{n,\epsilon_1,\epsilon_2}^y(s))} ds\Big]\\
&\quad +\sum_{y\in\mathbb{X}_n} {Q_{\delta_n,\epsilon_1}^{t;x-y}}\xi(y);
\quad (t,x)\in\mathbb{T}\times\mathbb{X}_n,
\end{split} \label{GSDDEeq}
\end{equation}
where $Q_{\delta_n,\epsilon_1}^{t:x}$ is the fundamental solution to the
parametrized deterministic heat equation on the lattice $\mathbb{X}_n$:
$$
\frac{du_{n,\epsilon_1}^x(t)}{dt}
= \frac{\epsilon_1}{2}\Delta_nu_{n,\epsilon_1}^x(t);
\quad (t,x)\in(0,T]\times\mathbb{X}_n.
$$
Furthermore, there is at least one weak solution to $\eqref{SDDE}$
$($and hence $\eqref{GSDDEeq})$.
\end{lem}

Similar to the zero drift case (Lemma 2.1 in \cite{ASDDE1}), the equivalence assertion in Lemma \ref{GSDDE} follows as in the continuous time-space case (e.g., Walsh \cite{WA}) from an equivalence of test function and Green function
formulations argument, and the existence is a straightforward generalization of standard SDEs arguments and the details will be omitted.

\begin{rem} \rm
Just as in the continuous time-case, where we can look probabilistically at the fundamental solution of the deterministic
heat equation as the density of Brownian motion, we note that $Q_{\delta_n,\epsilon_1}^{t;x}$ is the density of a symmetric
$1$-dimensional random walk on $\mathbb{X}_n$, in which the times between transitions are exponentially distributed with mean $\delta_n^2/\epsilon_1$,
for all $n$.  To simplify notations, we will suppress the dependence on the parameters $\epsilon_1,\epsilon_2$ unless we want to expressly
consider the effect of variations in them on the SPDE or SDDE (the subscript
$\delta_n$ in $Q^{t;x}_{\delta_n}$ is to remind us that the lattice points are $\delta_n$ apart).
Of course, by enlarging the filtration $\eus{F}_t$, we can accomodate random
initial $\xi$.    Also, if the space $\mathbb{R}$ is replaced by a closed
bounded interval $\mathbb{L}=[a,b]$, $a,b\in\mathbb{R}$
then $\mathbb{X}_n$ is replaced by $\mathbb{X}_n\cap\mathbb{L}$;
and the random walk will be either reflected or absorbed at $a\mbox{ and }b$,
with corresponding densities $Q_{\delta_n}$,
depending on whether we have Neumann or Dirichlet boundary conditions.
\label{Qasdensity}
\end{rem}

Using linear interpolation, we extend the definition of the already continuous-in-time process $\tilde{U}^{x}_n(t)$
on $\mathbb{T}\times\mathbb{X}_n$, so as to obtain a continuous process on $\mathbb{T}\times\mathbb{R}$, for each
$n\in\mathbb{N}$, which we will also denote by $\tilde{U}^{x}_n(t)$.    Henceforth,
any such sequence $\{\tilde U_n\}$ of interpolated $\tilde U_n$'s will be called a continuous
solution to $e_{\mbox{\tiny heat}}^{\mbox{\tiny SDDE}}(a,b,\xi,n)$.
We now give our definitions of SDDEs limit solutions to $e^{\epsilon_1,\epsilon_2}_{\mbox{\tiny heat}} (a,b,\xi)$ (of course solutions on $\mathbb{R}_+\times\mathbb{R}$ are defined in the
same way, replacing  $\mathbb{T}$ with $\mathbb{R}_+$).

\begin{defn} \rm
[SDDEs limits solutions to
$e^{\epsilon_1,\epsilon_2}_{\mbox{\tiny heat}} (a,b,\xi)$]
 We say that the random field
$U(t,x)$ is a continuous SDDE limit solution to
$e^{\epsilon_1,\epsilon_2}_{\mbox{\tiny heat}} (a,b,\xi)$ on
$\mathbb{T}\times\mathbb{R}$ iff there is a continuous solution
$\{\tilde{U}^{x}_n(t)\}$ to the SDDE system $e_{\mbox{\tiny
heat}}^{\mbox{\tiny SDDE}}(a,b,\xi,n)$ on a usual probability
space $(\Omega,\eus{F}\{\eus{F}_t\},\mathbb{P})$ and with respect
to a Brownian system
$\{W_n^x(t)\}_{(n,x)\in\mathbb{N}\times\mathbb{X}_n}$ such that
$U$ has $\mathbb{P}$-a.s. continuous paths on
$\mathbb{T}\times\mathbb{R}$, and $U$ is the limit or a
modification of the limit of $\tilde U_n$ (or of a subsequence
$\tilde{U}_{n_k}$) as $n\nearrow\infty$ ($k\nearrow\infty$). When
desired; the types of the solution and the limit are explicitly
stated (e.g., we say strong (weak) SDDEs weak, in probability,
$L^p$, or a.s.~limit solution to indicate that the solution to the
approximating SDDEs system is strong (weak) and that the limit of
the SDDEs is in the weak, in probability, $L^p$, or a.s.~sense,
respectively). We say that uniqueness in law holds if whenever $U$
and $V$ are SDDEs limit solutions, $U$ and $V$ have the same law.
\label{limitsolns}
\end{defn}

\begin{rem} \rm
Although in this article we restrict our treatment to the weak SDDEs weak limit solutions and weak uniqueness (in law),
Definition \ref{limitsolns} easily admits limits in any sense, not only those mentioned above,
as well as stronger uniqueness (pathwise).
\label{defnrem}
\end{rem}

Next, we use the existence in Lemma \ref{GSDDE} to show the existence of a SDDE limit
solution to $e^{\epsilon_1,\epsilon_2}_{\mbox{\tiny heat}} (a,b,\xi)$, and that this solution is a solution in the standard sense of satisfying the test function formulation
of $e^{\epsilon_1,\epsilon_2}_{\mbox{\tiny heat}} (a,b,\xi)$.  This complements the results in \cite{ASDDE1} by treating the case of continuous drift which grows no more
than linearly.

\begin{thm} [Weak existence, uniqueness, and regularity]
Fix $\epsilon_1$ and $\epsilon_2$.  If the conditions in \eqref{acond} hold for some constant $K>0$, then every sequence of continuous SDDEs solutions is tight in
$C(\mathbb{T}\times\mathbb{R};\mathbb{R})$ and we have a weak SDDE weak limit solution $U$ to
$e^{\epsilon_1,\epsilon_2}_{\mbox{\tiny heat}} (a,b,\xi)$.   Moreover, there is a continuous random field $Y,$ with the same law as $U,$ that satisfies the test function formulation \
of $e^{\epsilon_1,\epsilon_2}_{\mbox{\tiny heat}} (a,b,\xi)$:
\begin{equation}
\begin{split}
&(Y(t)-\xi,\varphi)-\frac{\epsilon_1}{2}\int_0^t(Y(s),\varphi'')ds-\int_0^t(b(Y(s)),\varphi)ds\\
&=  \epsilon_2\int_0^t \int_{\mathbb{R}}a(Y(s,x))\varphi(x)\eus{W}(dx,ds)
\end{split}
\label{TFF}
\end{equation}
for every $\varphi\in C_c^{\infty}(\mathbb{R};\mathbb{R})$, where $(\cdot,\cdot)$ denotes the scalar product on $L^2(\mathbb{R})$.  The continuous paths of $Y$ are H\"older $\gamma_s\in(0,\frac12)$ in space and H\"older $\gamma_t\in(0,\frac14)$ in
time and are $L^p$ bounded for every $p\ge2$.  If $a(u)=u^\gamma$, with $1/2\le\gamma\le1$ and
$b(u)=\sum_{i=1}^N c_iu^{\alpha_i}$ for constants $c_i\in\mathbb{R}$, $N\in\mathbb{N}$, and $1\ge\alpha_i\ge\gamma$, $i=1,\ldots,N;$ then
uniqueness in law holds for $e^{\epsilon_1,\epsilon_2}_{\mbox{\tiny heat}} (a,b,\xi)$ on $[0,T]\times[0,L]$, for any $T,L\in\mathbb{R}_+^2$, and hence the convergence to $U$ is along the whole
SDDEs sequence.
\label{ww}
\end{thm}

\begin{rem} \rm
In Theorem \ref{ww} we suppressed the dependence on the epsilons, since the results are given for fixed $\epsilon_1,\epsilon_2$; i.e.,
$U=U_{\epsilon_1,\epsilon_2}$ and $Y=Y_{\epsilon_1,\epsilon_2}$.
\label{fixedepsilon}
\end{rem}

The next result reveals that the behavior of $e^{\epsilon_1,\epsilon_2}_{\mbox{\tiny heat}} (a,b,\xi)$ as $\epsilon_1,\epsilon_2\searrow0$ is controlled by the ratio
${\epsilon_2}/{{ \epsilon_1}^{1/4}}$:
the solutions blow up in $L^2$ if ${\epsilon_2}/{{ \epsilon_1}^{1/4}}\nearrow\infty$ and they converge to the deterministic
$\epsilon_1\searrow0$ limit in $L^{2q}$ ($q\ge1$) if ${\epsilon_2}/{{ \epsilon_1}^{1/4}}\searrow0$.

\begin{thm} [Limits of $e^{\epsilon_1,\epsilon_2}_{\mbox{\tiny heat}} (a,b,\xi)$
as $\epsilon_1,\epsilon_2\searrow0$] \label{thm2}
$\mbox{i})$ Suppose that the conditions in \eqref{acond} hold, and that there are constants $K_l,K_L>0$ such that $K_l\le a(u)\le K_L$
for all $u\in\mathbb{R}$.  If
$\epsilon_1,\epsilon_2\searrow0$ (or $\epsilon_1,\epsilon_2\nearrow\infty$) such that the ratio
${\epsilon_2}/{{ \epsilon_1}^{1/4}}\to\infty;$
then,  $\sup_{0\le s\le T}\sup_{{x\in\mathbb{R}}}\mathbb{E} U_{\epsilon_1,\epsilon_2}^{2}(s,x)\nearrow\infty$, for any $T>0$ and for any
SDDEs weak limit solution $U_{\epsilon_1,\epsilon_2}$ to $e^{\epsilon_1,\epsilon_2}_{\mbox{\tiny heat}} (a,b,\xi)$
$\mbox{ii})$ If, in addition to the conditions in \eqref{acond}, $b$ is Lipschitz and $a$ is bounded; and if $U_{\epsilon_1}(t,x)$ is the solution to the
deterministic PDE obtained from $e^{\epsilon_1,\epsilon_2}_{\mbox{\tiny heat}} (a,b,\xi)$ by setting $a\equiv0$, and $U_{\epsilon_1,\epsilon_2}$ is a SDDEs weak limit solution to
$e^{\epsilon_1,\epsilon_2}_{\mbox{\tiny heat}} (a,b,\xi)$, then for every $q\ge1$
$$\sup_{ 0\le t\le T}\sup_{x\in\mathbb{R} }\mathbb{E}\left|U_{\epsilon_1,\epsilon_2}(t,x)-U_{\epsilon_1}(t,x)\right|^{2q}\longrightarrow0$$
as $\epsilon_1,\epsilon_2, \mbox{ and }\epsilon_2/\epsilon_1^{1/4}\to 0$.    Also, if $\left\{\tilde{U}^x_{n,\epsilon_1,\epsilon_2}(t)\right\}$ is a
solution to the SDDEs system $\left\{e_{\mbox{\tiny heat}}^{\mbox{\tiny SDDE}}(a,b,\xi,n)\right\}$ and if $\left\{\tilde{U}_{n,\epsilon_1}\right\}$ is a solution to the deterministic system
obtained from $e_{\mbox{\tiny heat}}^{\mbox{\tiny SDDE}}(a,b,\xi,n)$ by setting $a\equiv0$, then for every $q,n\ge1$
$$\sup_{ 0\le t\le T}\sup_{x\in\mathbb{R} }\mathbb{E}\left|\tilde{U}^x_{n,\epsilon_1,\epsilon_2}(t)-\tilde{U}_{n,\epsilon_1}(t,x)\right|^{2q}\longrightarrow0$$
as $\epsilon_1,\epsilon_2, \mbox{ and }\epsilon_2/\epsilon_1^{1/4}\to 0$.
\end{thm}

\begin{rem} \rm
Taking note of Remark \ref{fixedepsilon}, it follows from well known facts in Walsh \cite{WA} that, under \eqref{acond}, \eqref{TFF} is equivalent to the Green function formulation \eqref{GFF}
below
\begin{equation}
\begin{split}
&Y_{\epsilon_1,\epsilon_2}(t,x)-\int_{\mathbb{R}}G_{\epsilon_1}(t;x,y)\xi(y)dy\\&=\int_{\mathbb{R}}\int_0^t G_{\epsilon_1}(s,t;x,y)\Big\lbrack{\epsilon_2
a(Y_{\epsilon_1,\epsilon_2}(s,y))} \eus{W}(ds,dy)+ b(Y_{\epsilon_1,\epsilon_2}(s,y))dsdy\Big\rbrack,
\end{split}
\label{GFF}
\end{equation}
where
$$
G_{\epsilon_1}(t;x,y)=\frac{1}{\sqrt{2\epsilon_1\pi t}}e^{-(x-y)^2/2\epsilon_1t}.
$$
Also, note that when $\epsilon_1,\epsilon_2$ are fixed,
$\sup_{0\le s\le T}\sup_{{x\in\mathbb{R}}}\mathbb{E} U_{\epsilon_1,\epsilon_2}^{2}(s,x)<\infty $ for all $T>0$
(see Proposition \ref{Expbd},  \eqref{Lp}, and note that $Y$ has the same law as $U$).
\label{rm2}
\end{rem}

\section{Existence, uniqueness, and regularity}

The proof of Theorem \ref{ww} proceeds in several steps as in the heat SPDE
in \cite{ASDDE1}, with the extra difficulty
caused by the extra term $b(U)$: we first get
Kolmogorov type estimates on the spatial and temporal differences of the continuous
$\tilde{U}^{x}_n(t)$'s establishing tightness, and so by Lemma \ref{GSDDE} and Definition \ref{limitsolns} this implies the existence of a weak SDDE weak limit solution to $e^{\epsilon_1,\epsilon_2}_{\mbox{\tiny heat}} (a,b,\xi)$.
Then we show that the limit satisfies the test function formulation of $e^{\epsilon_1,\epsilon_2}_{\mbox{\tiny heat}} (a,b,\xi)$.   Since Theorem \ref{ww} is stated for fixed $\epsilon_1,\epsilon_2$, we suppress
the dependence on these parameters (except in Lemma \ref{2ndQinequality} below which we use in Remark \ref{SDDEno} in the third section), and we assume without
loss of generality that they are both $1$.
Throughout the article, $K$ will denote a constant that may change its value
from one step to the next.

\subsection{Random walk estimates}

The first set of estimates we need are bounds on the random walk density $Q^{t;x}_{\delta_n}$.  Since all the results in this section hold for all
$n$, we will suppress the dependence on $n$, except in \eqref{determ}, to simplify the notation.    The first three lemmas are taken directly from \cite{ASDDE1} p.~32 and are
reproduced below for convenience:

\begin{lem}
There is a constant $K$ such that
\[\sum_{x\in\mathbb{X}} {(Q^{t:x}_\delta)}^2 \leq K\delta/\sqrt t\quad\mbox{ and }\quad\sum_{x\in\mathbb{X}} {(Q_{\delta,\epsilon_1}^{t;x})}^2 \leq K\delta/\sqrt{\epsilon_1 t}, \]
and hence
\[\int_0^t\sum_{x\in\mathbb{X}} {(Q^{s;x}_\delta)}^2 \, ds
\leq K\delta\sqrt t \quad\mbox{ and }\quad\int_0^t\sum_{x\in\mathbb{X}} {(Q_{\delta,\epsilon_1}^{s;x})}^2 \, ds
\leq K\delta\sqrt{\frac{t}{\epsilon_1}}.\]
\label{2ndQinequality}
\end{lem}

\begin{lem}
There is a constant $K$ such that
\[\int_0^t\sum_{x\in\mathbb{X}} {(Q^{s;x}_\delta - Q^{s;x+z}_\delta)}^2 ds \leq K\delta |z|.\]
\label{3rdQinequality}
\end{lem}

\begin{lem}
There is a constant $K$ such that
\[\int_0^t\sum_{x\in\mathbb{X}} {(Q^{t-s;x}_\delta - Q^{r-s;x}_\delta)}^2 ds \leq K\delta\sqrt
 {t-r},\]
for $r<t$, and with the convention that $Q_{\delta}^x(t)=0$ if $t<0$.
\label{4thQinequality}
\end{lem}

The next three Lemmas are needed to deal with the reaction term $b(U)$ in the SPDE.

\begin{lem}
There is a constant $K$ such that
\begin{equation}
\sum_{x\in\mathbb{X}} {|Q^{t;x}_\delta - Q^{t:x+z}_\delta|} \leq 1 \wedge K \frac{|z|}{\sqrt{t}},
\label{ineq1}
\end{equation}
and thus
\begin{equation}
\int_0^t\sum_{x\in\mathbb{X}} \left|{Q^{s;x}_\delta - Q^{s;x+z}_\delta}\right| ds \leq K \sqrt{t} |z|.
\label{ineq2}
\end{equation}
\label{5thinequality}
 \end{lem}

\begin{proof}
We use a standard maximal coupling argument.  For details on coupling and related
techniques the interested reader could consult \cite{HT} and the references therein.

Denote by $\Lambda_\delta^{x}(t)$
the law of the random walk, starting at $x$, after time $t$; and let $\| \cdot\|$ denote the total variation norm.  Then we are trying to bound
  $$A\overset{\triangle}{=}   \int_0^t \left\|\Lambda_\delta^{x}(s) - \Lambda_\delta^{x+z}(s)\right\|ds$$
So, using a coupling argument, we start two copies of the random walk at $0$ and $z$ and run the
maximal coupling.  The total variation at time $t$ between the two random walks laws
$\|\Lambda_\delta^{0}(t)-\Lambda_\delta^{z}(t)\|$ is exactly the probability of not yet coupling,
which is readily seen to be at most $1 \wedge K |z|/t^{1/2}$, for some constant $K$.   Thus, we get the first inequality
\eqref{ineq1}, and the second inequality \eqref{ineq2} (the bound on $A$) immediately follows.
\end{proof}

\begin{lem}
There are constants $K$ and $\delta^*$ such that
$$\sum_{x\in\mathbb{X}} \left|{Q^{t;x}_\delta - Q^{s;x}_\delta}\right| \leq
 1 \wedge K \frac{\sqrt{t}-\sqrt{s}}   {\sqrt{t}},
 $$
whenever $s<t$ and $\delta<\delta^*$, and hence
$$
\int_0^t\sum_{x\in\mathbb{X}} \left|{Q^{t-s;x}_\delta - Q^{r-s;x}_\delta}\right|ds
\leq K\big(1 + \log\big[\frac{t}{t-r}\big]\big){(t-r)},
$$
  for $r<t$ and $\delta<\delta^*$, and with the convention that $Q_{\delta}^{t;x}=0$ if $t<0$.
\label{6thinequality}
\end{lem}

This is a central limit theorem type argument, which we briefly present
here for convenience.

\begin{proof}
To see the first inequality, note that the total variation distance is
\begin{equation}
\left\|\Lambda_\delta^{x}(s)- \Lambda_\delta^{x}(t)\right\|=\sup_k\left|\Lambda_\delta^{x}(s){\left([-k,k]\right)} -\Lambda_\delta^{x}(t){\left([-k,k]\right)}\right|,
\label{ttlvdist}
\end{equation}
where $\Lambda_\delta^{x}(t)$ is as in the proof of Lemma \ref{5thinequality}.
By the central limit theorem, we see that the limit as $\delta \to 0$ in \eqref{ttlvdist} is bounded by a constant multiple of
$(t^{1/2} - s^{1/2}) / t^{1/2}$; then there is a $\delta^*>0$ such that, whenever $\delta<\delta^*$, the total variation distance in \eqref{ttlvdist} is bounded
by a constant multiple of $(t^{1/2} - s^{1/2}) / t^{1/2}$ as well, with a possibly different constant.
So, for small enough $\delta$, we get the first inequality.   The second inequality follows upon using the fact that
$(t^{1/2} - s^{1/2}) / t^{1/2} \leq (t-s) / t$ and integrating, using the convention that $Q_{\delta}^{t;x}=0$ if $t<0$.
\end{proof}

 From this point on, and without explicitly stating it, we will assume
 $\delta<\delta^*$ whenever needed.  Consequently, we have the following statement.

\begin{lem}
There is a constant $K$, depending only on $T$, such that
$$
\int_0^t\sum_{x\in\mathbb{X}} \left|{Q^{t-s;x}_\delta - Q^{r-s;x}_\delta}\right|ds
\le K\big[(t-r)+ (t-r)^{1-e^{-1}}\big],
$$
 for $0\le r<t\le T$, and with the convention that $Q_{ \delta}^{t;x}=0$ if $t<0$.
 \label{7thinequality}
 \end{lem}

\begin{proof}  In light of Lemma \ref{6thinequality}, it suffices to show that
\begin{equation}
\log\big[\frac{T\vee1}{t-r}\big]\le (T\vee1){(t-r)}^{-e^{-1}},
\label{e0}
\end{equation}
for  $0\le r<t\le T$.  So, setting $x=t-r$ and letting
$\ell(x)\overset{\triangle}{=} \log\left[{(T\vee1)}/{x}\right]-(T\vee1){x}^{-e^{-1}}$, we see from an easy calculus computation that
$$
\max\ell(x)=\ell([{(T\vee1)}/{e}]^e)=(1-e)\log[T\vee1]\le0,
$$
proving \eqref{e0}.
\end{proof}

\subsection{Bounds on moments of $\tilde{U}^x(t)$}

The main result of this subsection is as follows.

\begin{prop} \label{Expbd}
There exists a constant $K$ depending only on $q$, $\max_x|\xi(x)|$, and $T$ such that
\begin{equation*}
M_q(t) \leq K\exp{\{Kt\}}; \quad \forall\ 0\le t\le T, \ q\ge1,
\end{equation*}
where $M_q(t) = \sup_x \mathbb{E}|\tilde{U}^x(t)|^{2q}$.   In particular, $M_q$ is bounded on $\mathbb{T}$ for all $q\ge1,$ and
\begin{equation}
\sup_{t\in\mathbb{T}} M_q(t)\le K\exp{\{KT\}}.
\label{bddMts}
\end{equation}
\end{prop}

 The proof of Proposition \ref{Expbd} proceeds via the following lemma and its
 corollary.

\begin{lem} There exists a constant $K$ depending only on $q\ge1$, $\max_x |\xi(x)|$, and $T$ such that
\begin{equation*}
M_q(t) \leq K \Big(1 + \int_0^t \big[\frac{M_q(s)}{\sqrt{t - s}}+M_q(s)\big]ds\Big);
  \quad \forall0\le t\le T,\ q\ge1.
\end{equation*}
\label{1stboundonUtilde}
\end{lem}

\begin{proof}
Fix $q\ge1$, let $\tilde{U}_D^{x}(t)\overset{\triangle}{=} \sum_{y\in\mathbb{X}} Q^{t;x,y}_{\delta_n}\xi(y)$ (the deterministic part of $\tilde{U}$).  Then,  for any $(t,x)\in\mathbb{T}\times\mathbb{X}$, we have:
\begin{equation}
\begin{split}
\mathbb{E}|\tilde{U}^x(t)|^{2q}& = \mathbb{E}
\Big| \sum_{y\in {\mathbb{X}}}\int_0^t Q^{t-s; x,y}_\delta
\big[\frac{a(\tilde{U}^y(s))}{\sqrt{\delta}} dW^{y}(s)
+ b(\tilde{U}^y(s))ds\big]+ \tilde{U}_D^{x}(t) \Big|^{2q} \\
& \leq K\Big(\mathbb{E}\Big|\sum_{y\in {\mathbb{X}}}\int_0^t
Q^{t-s; x,y}_\delta\frac{a(\tilde{U}^y(s))}{\sqrt{\delta}} dW^{y}(s) \Big|^{2q}
+\left|\tilde{U}_D^{x}(t) \right|^{2q}\Big)\\
&+ K\mathbb{E}\Big|\sum_{y\in\mathbb{X}}\int_0^t
Q^{t-s; x,y}_\delta b(\tilde{U}^y(s))ds\Big|^{2q}.
 \end{split}
 \label{befBurkholder}
 \end{equation}
Applying Burkholder inequality to
\[
V^x(t) = \sum_{y\in {\mathbb{X}}}\int_0^t Q^{t-s; x,y}_\delta
\frac{a(\tilde{U}^y(s))} {\sqrt{\delta}}dW^{y}(s)
\]
Reduces \eqref{befBurkholder} to
\begin{equation}
\begin{split}
\mathbb{E}|\tilde{U}^x(t)|^{2q}&\leq K\Big(\mathbb{E}\Big|\sum_{y\in {\mathbb{X}}}
\int_0^t {(Q^{t-s; x,y}_\delta)}^2\frac{a^2(\tilde{U}^y(s))}{\delta}ds\Big|^q
+ |\tilde{U}_D^{x}(t)|^{2q}\Big) \\
&\quad +K\mathbb{E}\Big|\sum_{y\in\mathbb{X}}
\int_0^t Q^{t-s; x,y}_\delta b(\tilde{U}^y(s))ds\Big|^{2q}.
\end{split}
\label{afterBurkholder}
\end{equation}
Now, for a fixed point
$(t,x)\in\mathbb{T}\times\mathbb{X}$ let $\mu_t^x$ and $\nu_t^x$ be the measures
on $[0,t] \times\mathbb{X}$ defined by $d\mu_t^x(s,y)
=({(Q^{t-s; x,y}_\delta)}^2/\delta) ds$ and $d\nu_t^x(s,y)={Q^{t-s; x,y}_\delta}ds$,
 and let $|\mu_t^x| = \mu_t^x([0,t] \times \mathbb{X})$ and
$|\nu_t^x| = \nu_t^x([0,t] \times \mathbb{X})$. Then, we can rewrite \eqref{afterBurkholder} as
 \begin{equation}
\begin{split}
 \mathbb{E}|\tilde{U}^x(t)|^{2q}
\leq &K\Big(\mathbb{E}\left|\int_{[0,t] \times \mathbb{X}}a^2(\tilde{U}^y(s))
\frac{d\mu_t^x(s,y)}{|\mu_t^x|}\right|^q {|\mu_t^x|}^q +|\tilde{U}_D^{x}(t)|^{2q}
\Big)\\
 &+ K \mathbb{E}\Big|\int_{[0,t] \times \mathbb{X}}b(\tilde{U}^y(s))
 \frac{d\nu_t^x(s,y)}{|\nu_t^x|}\Big|^{2q} {|\nu_t^x|}^{2q}.
\end{split}
\label{afterBurkholder2}
\end{equation}
Observing that $\mu_t^x/|\mu_t^x|$ and $\nu_t^x/|\nu_t^x|$ are probability measures, we apply Jensen's inequality, the growth condition on $a$ and $b$,
and other elementary inequalities to \eqref{afterBurkholder2} to obtain
\begin{align*}
\mathbb{E}|\tilde{U}^x(t)|^{2q}
&\leq K\Big(\mathbb{E}\Big[\int_{[0,t] \times \mathbb{X}}\left|a(\tilde{U}^y(s))
\right|^{2q}\frac{d\mu_t^x(s,y)}{|\mu_t^x|}\Big]
{|\mu_t^x|}^q + |\tilde{U}_D^{x}(t)|^{2q}\Big)\\
&\quad +K\mathbb{E}\Big[\int_{[0,t] \times \mathbb{X}}
\left|b(\tilde{U}^y(s))\right|^{2q}\frac{d\nu_t^x(s,y)}{|\nu_t^x|}\Big]
{|\nu_t^x|}^{2q}\\
& \leq K\Big[\int_{[0,t] \times \mathbb{X}}
\left(1 + \mathbb{E}|\tilde{U}^y(s)|^{2q}\right){d\mu_t^x(s,y)}\Big]
{|\mu_t^x|}^{q-1} + K |\tilde{U}_D^{x}(t)|^{2q} \\
&\quad +K \Big[\int_{[0,t] \times \mathbb{X}}\left(1 + \mathbb{E}|
\tilde{U}^y(s)|^{2q}\right)d\nu_t^x(s,y)\Big]{|\nu_t^x|}^{2q-1}\\
&= K\Big(\Big[\sum_{y\in {\mathbb{X}}}\int_0^t\frac{(Q^{t-s; x,y}_\delta)^2}
{\delta}\left(1 + \mathbb{E}|\tilde{U}^y(s)|^{2q}\right)ds\Big]
{|\mu_t^x|}^{q-1}+|\tilde{U}_D^{x}(t)|^{2q} \Big)\\
&\quad +K\Big[\sum_{y\in {\mathbb{X}}}\int_0^t{Q^{t-s; x,y}_\delta}
\left(1 + \mathbb{E}|\tilde{U}^y(s)|^{2q}\right)ds\Big]{|\nu_t^x|}^{2q-1}.
\end{align*}
Using the simple fact that $\sum_{y\in {\mathbb{X}}}{Q^{t-s; x,y}_\delta}=1$
and Lemma \ref{2ndQinequality}, we see that  $|\nu_t^x|$ and $|\mu_t^x|$ are
uniformly bounded for $t\leq T$.
So, using the boundedness of $\xi$, and hence of $\tilde{U}_D^{x}(t)$,
Lemma \ref{2ndQinequality} and the definition of $M_q(s)$, we get
\begin{align*}
\mathbb{E}|\tilde{U}^x(t)|^{2q}
&\leq K\Big(1+\sum_{y\in {\mathbb{X}}}
\int_0^t\big[\frac{(Q^{t-s; x,y}_\delta)^2}{\delta}M_q(s)+{Q^{t-s; x,y}_\delta}
M_q(s)\big]ds\Big)\\
&\overset{R_1}{\leq} K \Big(1 + \int_0^t\big[\frac{M_q(s)}{\sqrt{t - s}}+M_q(s)\big]ds
\Big).
\end{align*}
Here, $R_1$ follows from Lemma \ref{2ndQinequality} and the fact that
$\sum_{y\in {\mathbb{X}}}{Q^{t-s; x,y}_\delta}=1$. This implies that
\begin{equation*}
M_q(t) \leq K \Big(1 + \int_0^t\big[\frac{M_q(s)}{\sqrt{t - s}}+M_q(s)\big]ds\Big).
\end{equation*}
\end{proof}

\begin{cor}
There exists a constant $K$ depending only on $q$, $\max_x |\xi(x)|$, and $T$
such that
\begin{equation*}
M_q(t) \leq K\Big(1+ \int_0^t M_q(s)ds \Big);\quad 0\le t\le T,\ q\ge1.
\end{equation*}
\label{befG}
\end{cor}

\begin{proof}
Iterating the bound in Lemma \ref{1stboundonUtilde} once, and changing the
order of integration, we obtain
\begin{equation}
\begin{split}
 M_q(t) \leq & K\Big\{1 + K\Big[\int_0^t\big(\frac{1}{\sqrt{t-s}}\big) ds \\
 &+\int_0^t{M_q(r)}\Big(\int_r^t\frac{1}{\sqrt{t - s}\sqrt{s-r}}ds
 +\int_r^t\frac{1}{\sqrt{t - s}}ds\Big)dr\\
 &+\int_0^tM_q(r)\Big(\int_r^t\frac{1}{\sqrt{t - s}}ds\Big)dr\Big]
 +\int_0^tM_q(s)ds\Big\}\\
 \le& K\Big(1+ \int_0^t M_q(s)ds \Big).
\end{split} \label{mom1}
\end{equation}
\end{proof}

Now the proof of Proposition \ref{Expbd} is a straightforward application of
 Gronwall's lemma to Corollary \ref{befG}.

\subsection{Bounds on spatial and temporal differences moments and tightness
of the approximating SDDEs}

Let $\tilde{U}^x(t)= \tilde{U}_R^x(t) + \tilde{U}_D^{x}(t)$, where $\tilde{U}_R^x(t)$ denotes the first two (random) terms on the r.h.s. of \eqref{GSDDEeq}.
It is easily seen (see \cite{SZ}) that the deterministic part $\tilde{U}_D^{x}(t)$ converges pointwise to the solution of the deterministic heat equation as $n\to\infty$
($\delta_n\to0$); i.e.,
\begin{equation}
\lim_{n\to\infty}\tilde{U}_{n,D}^x(t)= \int_{\mathbb{R}} G(t;x,y)\xi(y)dy,\hspace{0.5cm} \forall (t,x)\in\eus{R}_{T}.
\label{determ}
\end{equation}
So, to show weak convergence of  a subsequence of $\tilde{U}^x(t)$; it is enough to show tightness, and hence the weak convergence, of the random part $ \tilde{U}_R^x(t)$.
Using the inequalities of the previous two subsections, we obtain

\begin{lem}[Spatial differences]
There exists a constant $K$ depending only on $q$, $\max_x |\xi(x)|$, and $T$
such that
\[
\mathbb{E}\left| \tilde{U}_R^x(t) - \tilde{U}_R^y(t)\right|^{2q}
\le K\left(|x-y|^{q}+|x-y|^{2q}\right),
\]
for all $x,y \in \mathbb{X}$ and $t\in\mathbb{T}$.
\label{boundsonsdifferences}
\end{lem}

\begin{proof}
Using Burkholder inequality, we have for any $(t,x,y)\in\mathbb{T}\times\mathbb{X}^2$
\begin{equation}
\begin{split}
\mathbb{E}\Big| \tilde{U}_R^x(t) - \tilde{U}_R^y(t)\Big|^{2q}
\leq& K\mathbb{E}{\Big|\sum_{z\in\mathbb{X}}\int_0^t
\frac{{\left(Q^{t-s; x,z}_\delta- Q^{t-s;y,z}_\delta\right)}^2}{\delta}a^2
(\tilde{U}^{z}(s)) ds\Big|}^q\\
&+K\mathbb{E}{\Big|\sum_{z\in\mathbb{X}}\int_0^t{\left(Q^{t-s; x,z}_\delta
-Q^{t-s;y,z}_\delta\right)}{b(\tilde{U}^{z}(s))} ds\Big|}^{2q}.
\end{split}
\label{spdiff1}
\end{equation}
For any fixed but
arbitrary point $(t,x,y)\in\mathbb{T}\times\mathbb{X}^2$ let $\mu_t^{x,y}$ and $\nu_t^{x,y}$ be the measures defined on $[0,t]\times\mathbb{X}$ by
$d\mu_t^{x,y}(s,z)=({(Q^{t-s; x,z}_\delta-Q^{t-s;y,z}_\delta)}^2/\delta) ds$ and $d\nu_t^{x,y}(s,z)=\left|Q^{t-s; x,z}_\delta-Q^{t-s;y,z}_\delta\right|ds$, and let
$|\mu_t^{x,y}| = \mu_t^{x,y}([0,t] \times \mathbb{X})$ and $|\nu_t^{x,y}| = \nu_t^{x,y}([0,t]\times\mathbb{X})$.    So, from \eqref{spdiff1}, Jensen's inequality,
the growth condition on $a$ and $b$, the definition of $M_q(t)$, and elementary inequalities, we have
\begin{equation}
\begin{split}
\mathbb{E}\left| \tilde{U}_R^x(t) - \tilde{U}_R^y(t)\right|^{2q}
\le& K\mathbb{E}\Big[\int_{[0,t] \times \mathbb{X}}
\left|a(\tilde{U}^{z}(s))\right|^{2q}\frac{d\mu_t^{x,y}(s,z)}{|\mu_t^{x,y}|}\Big]
{|\mu_t^{x,y}|}^q \\
&+K\mathbb{E}\Big[\int_{[0,t] \times \mathbb{X}}\left|{b(\tilde{U}^{z}(s))}
\right|^{2q}\frac{d\nu_t^{x,y}(s,z)}{|\nu_t^{x,y}|}\Big]{|\nu_t^{x,y}|}^{2q}\\
\le&  K\Big[\int_{[0,t] \times \mathbb{X}}\left (1+M_q(s)\right)
\frac{d\mu_t^{x,y}(s,z)}{|\mu_t^{x,y}|}\Big] {|\mu_t^{x,y}|}^q\\
&+K\Big[\int_{[0,t] \times \mathbb{X}}\left(1+M_q(s)\right)
\frac{d\nu_t^{x,y}(s,z)}{|\nu_t^{x,y}|}\Big]{|\nu_t^{x,y}|}^{2q}
\end{split}
\label{spdiff2}
\end{equation}
Now, using the boundedness of $M_q$ on $\mathbb{T}$ (Proposition \ref{Expbd}),
we get
 \begin{equation*}
\begin{split}
\mathbb{E}\left| \tilde{U}_R^x(t) - \tilde{U}_R^y(t)\right|^{2q}\le K\left({|\mu_t^{x,y}|}^q+{|\nu_t^{x,y}|}^{2q}\right)\le K\left(|x-y|^{q}+|x-y|^{2q}\right),
\end{split}
\label{spdiff3}
\end{equation*}
where the last inequality follows from Lemma \ref{3rdQinequality} and Lemma
\ref{5thinequality}.
\end{proof}

\begin{lem}[Temporal differences] There exists a constant $K$ depending only on $q$, $\max_x|\xi(x)|$, and
$T$ such that
\[\mathbb{E}\left| \tilde{U}_R^x(t) - \tilde{U}_R^x(r) \right|^{2q} \leq K \left(|t - r|^{q/2}+|t-r|^{2q}+|t-r|^{2q{(1-e^{-1}})}\right),\] for all $x \in \mathbb{X}$ and for all $t,r \in\mathbb{T}$.
\label{boundsontdifferences}
\end{lem}

\begin{proof}
Assume without loss of generality  that $r<t$.  For a fixed point $(r,t,x)$,
let $\mu_{t,r}^x$ and $\nu_{t,r}^x$ be the measures defined on
$[0,t]\times\mathbb{X}$ by
\begin{gather*}
d\mu_{t,r}^x(s,z)=({(Q^{t-s; x,z}_\delta-Q^{r-s;x-z}_\delta)}^2/\delta) ds\\
d\nu_{t,r}^x(s,z)=\left|Q^{t-s; x,z}_\delta-Q^{r-s;x-z}_\delta\right|ds,
\end{gather*}
with the convention that $Q_{\delta}^{t;x}=0$ if $t<0$, and let $|\mu_{t,r}^x| = \mu_{t,r}^x([0,t] \times \mathbb{X})$ and $|\nu_{t,r}^x| = \nu_{t,r}^x([0,t]\times\mathbb{X})$.
Then, arguing as in Lemma \ref{boundsonsdifferences}, we obtain
 \begin{equation*}
 \begin{split}
\mathbb{E}\left| \tilde{U}_R^x(t) - \tilde{U}_R^x(r)\right|^{2q}&\le K\left({|\mu_{t,r}^x|}^q+{|\nu_{t,r}^x|}^{2q}\right)\\
&\le K\left((t-r)^{q/2}+(t-r)^{2q}+(t-r)^{2q{(1-e^{-1}})}\right),
\end{split}
\end{equation*}
where the last inequality follows from Lemma \ref{4thQinequality} and Lemma \ref{7thinequality}.
\end{proof}

Following \cite{ASDDE1}, we conclude from Lemma \ref{boundsonsdifferences} and
Lemma \ref{boundsontdifferences} that the random part of
the sequence of (interpolated) continuous SDDEs solutions
${\left(\Tilde{U}_n^x(t)\right)}_{n=1}^{\infty}$ is tight on
$C(\mathbb{T}\times {\mathbb{R}};\mathbb{R})$.   This and \eqref{determ} imply
that there exists a
weakly convergent subsequence $\Tilde{U}_{n_k}$.  So, by Lemma \ref{GSDDE} and
Definition \ref{limitsolns} we have proven the existence of a weak SDDE weak
limit solution to $e^{\epsilon_1,\epsilon_2}_{\mbox{\tiny heat}} (a,b,\xi)$.
Then, following Skorokhod, we can construct processes
$Y_k \overset{d}{=} \Tilde{U}_{n_k}$ on some
probability space $(\Omega^S,\eus{F}^S, \mathbb{P}^S)$ such that with probability
$1$, as $k\to\infty$, $Y_k(t,x)$ converges to a random field $Y(t,x)$ uniformly
 on compact subsets of $\mathbb{T}\times{\mathbb{R}}$ for any $T$.

We will now show that $Y(t,x)$ is a solution to  $e^{\epsilon_1,\epsilon_2}_{\mbox{\tiny heat}} (a,b,\xi)$ in the traditional sense, by showing that it solves an equivalent
martingale problem to the test function formulation \eqref{TFF} for $e^{\epsilon_1,\epsilon_2}_{\mbox{\tiny heat}} (a,b,\xi)$ (see Theorem \ref{1stmartprob} and Theorem \ref{marttff}), and this will complete the
proof of the existence assertions in Theorem \ref{ww}.   It is worth noting that Theorem \ref{marttff} eliminates the need for a second
martingale problem (as in Theorem 5.3 in \cite{ASDDE1}), and provides a simpler way to establish the equivalence to the test
function formulation of any heat-based SPDE (not just for reaction diffusions SPDEs).

\subsection{The Martingale problem}

For every $\varphi\in C_c^{\infty}(\mathbb{R};\mathbb{R})$ let
\begin{equation}
\begin{split}
&S^\varphi(Y_k,t)\\
=&\sum_{x\in{\mathbb X}_{n_k}}\left[Y_k(t,x)-\xi(x)\right] \varphi(x)\,\delta_{n_k}-
\frac12\int_0^t\sum_{x\in{\mathbb X}_{n_k}}Y_k(s,x)\,\Delta_{n_k} \varphi(x)\,\delta_{n_k}ds\\
&-\int_0^t\sum_{x\in{\mathbb{X}}_{n_k}}b(Y_k(s,x))\varphi(x) \delta_{n_k}ds
\label{orgsum}
\end{split}
\end{equation}
and let $\eus{G}_t$ be the filtration on $(\Omega^S,\eus{F}^S, \mathbb{P}^S)$ generated by the process
$S^\varphi(Y_k,t)$ for all $\varphi$ and all $k$; i.e., $\eus{G}_t=
\sigma\left[S^\varphi(Y_k,s);0\le s\le t,
\varphi\in C_c^{\infty}(\mathbb{R};\mathbb{R}), k=1,2,\cdots\right]$.

\begin{thm} For $\varphi\in C_c^{\infty}(\mathbb{R};\mathbb{R})$ we have
\begin{enumerate}
\renewcommand{\labelenumi}{(\roman{enumi})}
\item $\{M^\varphi(t), \eus{G}_t\}$ is a martingale, for every
$\varphi\in C_c^{\infty}(\mathbb{R};\mathbb{R})$, where
$$M^\varphi(t)\overset{\triangle}{=}  (Y(t)-\xi,\varphi) - \frac12\int_0^t (Y(s),\varphi'') ds-\int_0^t(b(Y(s)),\varphi)ds;\ 0\le t<T,$$
where $(\cdot,\cdot)$ denotes the scalar product on $L^2(\mathbb{R})$,
\item ${\langle{M^\varphi(\cdot)}\rangle_t=\langle {(Y,\varphi)} \rangle_t = \int_0^t \int_{{\mathbb{R}}}a^2(Y^x(s))\varphi^2(x)dx ds}$
\end{enumerate}
\label{1stmartprob}
\end{thm}

\begin{proof}
(i)\quad  Assume that the sequence of Brownian motions $\Tilde{W}_n^x(t)$ in \eqref{SDDE}
is defined on some probability space $(\Omega,\eus{F},\mathbb{P})$ and adapted to a filtration
$\{{\eus{F}}_t\}_{t\ge0}$.  We first observe from \eqref{SDDE} that for any $k$
$$\left[\tilde{U}^x_{n_k}(t)-\xi(x)\right] \varphi(x)\delta_{n_k}- \int_0^t \left[\frac12\Delta_{n_k} \tilde{U}^x_{n_k}(s)+b(\tilde{U}^x_{n_k}(s))\right]\varphi(x) \delta_{n_k}ds$$
is an ${\eus{F}}_t$-martingale for each $x\in {\mathbb{X}}_{n_k}$ (this easily follows from the growth condition on $a$ ((b) in \eqref{acond})
and Proposition \ref{Expbd},  along with the boundedness of $\varphi$).  Now since $\varphi$ has a compact support, it follows that
\begin{equation}
\begin{split}
&\sum_{x\in{\mathbb{X}}_{n_k}}\left[\tilde{U}^x_{n_k}(t)-\xi(x)\right]\varphi(x)\delta_{n_k}
 -\frac12\int_0^t\sum_{x\in{\mathbb{X}}_{n_k}}\Delta_{n_k}\tilde{U}^x_{n_k}(s)\varphi(x) \delta_{n_k}ds\\
 &-\int_0^t\sum_{x\in{\mathbb{X}}_{n_k}} b(\tilde{U}^x_{n_k}(s))\varphi(x) \delta_{n_k}ds\\
= &\sum_{x\in{\mathbb{X}}_{n_k}}\left[\tilde{U}^x_{n_k}(t)-\xi(x)\right] \varphi(x)\delta_{n_k}
-  \frac12\int_0^t\sum_{x\in{\mathbb{X}}_{n_k}}\tilde{U}^x_{n_k}(s)\Delta_{n_k} \varphi(x)\delta_{n_k}ds\\
&\quad-\int_0^t\sum_{x\in{\mathbb{X}}_{n_k}}b(\tilde{U}^x_{n_k}(s))\varphi(x) \delta_{n_k}ds\overset{\triangle}{=}  S^\varphi(\tilde{U}_{n_k},t)
\end{split}
\label{sumofsubseq}
\end{equation}
is a finite sum, and hence an ${\eus{F}}_t$-martingale.   Replacing the $\tilde{U}^x_{n_k}(t)$ in \eqref{sumofsubseq} by the $Y_k(t,x)$,
and letting $k\to\infty$,
we get that $S^\varphi(Y_k,t)\to M^\varphi(t)$  a.s. (uniformly on $\mathbb{T}$).  In
addition,  $S^\varphi(Y_k,t)$ are
uniformly integrable for each $t$ and each $\varphi$.  To see this, observe that
for each $t\in\mathbb{T}$ and each
$\varphi\in C_c^{\infty}(\mathbb{R};\mathbb{R})$, we have
\begin{equation*}
\begin{split}
\mathbb{E}|S^\varphi(Y_k,t)|^2&=\mathbb{E}|S^\varphi(\tilde{U}_{n_k},t)| ^2\\
&=\mathbb{E}\Big|\int_0^t\sum_{x\in{\mathbb{X}}_{n_k}}a(\tilde{U}^x_{n_k}(s))
\varphi(x)\sqrt{\delta_{n_k}}dW^x_{n_k}(s)\Big|^2\\
&\leq K \int_0^t\sum_{x\in{\mathbb{X}}_{n_k}}\mathbb{E} a^2(\tilde{U}^x_{n_k}(s))\varphi^2(x){\delta_{n_k}}ds\le K <\infty,
\end{split}
\end{equation*}
for some constant $K>0$ independent of $k$, where the last two inequalities follow from Burkholder's inequality, the boundedness and compact supportedness
of $\varphi$,  the growth condition on $a$ ((b) in \eqref{acond}), and Proposition \ref{Expbd}.  Thus, uniform integrability
of the sequence $\{S^\varphi(Y_k,t)\}_k$ follows for each $\varphi$ and each $t$.   So, If $s < t$
\begin{equation*}
\begin{split}
&\mathbb{E}\left[ M^\varphi(t) - M^\varphi(s)\left| \eus{G}_s\right.\right]
 = \lim_{k\to\infty}\mathbb{E}\left[S^\varphi(Y_k,t) - S^\varphi(Y_k,s)\left| \eus{G}_s\right.\right] = 0.
\end{split}
\end{equation*}
This proves (i).

\noindent (ii)\quad From \eqref{SDDE} it follows that
\begin{equation*}
\begin{split}
& d \Big[\sum_{x\in {\mathbb{X}}_{n_k}}\tilde{U}^x_{n_k}(t)\varphi(x)\delta_{n_k}
\Big]\\
&= \sum_{x\in {\mathbb{X}}_{n_k}} a\left(\tilde{U}^x_{n_k}(t)\right)\varphi(x) \sqrt{\delta_{n_k}}dW^x_{n_k}(t) \\
&\quad+\Big[\sum_{x\in {\mathbb{X}}_{n_k}}\big(\frac12 \Delta_{n_k}\tilde{U}^x_{n_k}
(t)+b(\tilde{U}^x_{n_k}(t))\big)\varphi(x)\delta_{n_k}\Big]dt\\
&= \Big[\frac12\sum_{x\in {\mathbb{X}}_{n_k}}\tilde{U}^x_{n_k}(t)
\Delta_{n_k}\varphi(x)\delta_{n_k}
+ \sum_{x\in {\mathbb{X}}_{n_k}}b(\tilde{U}^x_{n_k}(t))\varphi(x)\delta_{n_k}\Big]dt\\
&+\sum_{x\in {\mathbb{X}}_{n_k}}  a\left(\tilde{U}^x_{n_k}(t)\right)\varphi(x) \sqrt{\delta_{n_k}}dW^x_{n_k}(t).
\end{split}
\end{equation*}
Observing that the first two terms on the right hand side of the last equality in the above equation are of bounded variation, and that the
$\left(W^x_{n_k}(t)\right)_{x\in {\mathbb{X}}_{n_k}}$ is a sequence of independent Brownian motions, we obtain, after inspecting \eqref{sumofsubseq}, that
\begin{equation}
\Big\langle S^\varphi(\tilde{U}_{n_k},\cdot)\Big\rangle_t
=\Big\langle{\sum_{x\in {\mathbb{X}}_{n_k}}\tilde{U}^x_{n_k}(\cdot)
\varphi(x)\delta_{n_k}}\Big\rangle_t
= \int_0^t \Big[\sum_{x\in {\mathbb{X}}_{n_k}}a^2\left(\tilde{U}^x_{n_k}(s)\right)
\varphi^2(x)\delta_{n_k}\Big]ds
\label{QVoftheutilde}
\end{equation}
Again, replacing the $\tilde{U}^x_{n_k}(t)$ in \eqref{QVoftheutilde} by the $Y_k(t,x)$,
we get, for $0\le r\le t\le T$,
\begin{equation}
\mathbb{E}\left[\left(S^\varphi(Y_k,t)-S^\varphi(Y_k,r)\right)^2
\big|\eus{G}_r\right]
=\mathbb{E}\Big[\int_r^t\sum_{x\in {\mathbb{X}}_{n_k}}a^2\left(Y_k(s,x)\right)
\varphi^2(x)\delta_{n_k}ds\Big|\eus{G}_r\Big].
\label{expl}
\end{equation}
Again, we observe that $\left(S^\varphi(Y_k,t)-S^\varphi(Y_k,r)\right)^2$ are uniformly integrable, for each $r$ and $t$  and each $\varphi$.  To see that,
fix $p\ge1$, $0\le r\le t\le T$, and
$\varphi\in C_c^{\infty}(\mathbb{R};\mathbb{R})$; and apply Burkholder's
inequality to obtain
\begin{equation}
\begin{split}
 \mathbb{E}\left|S^\varphi(Y_k,t)-S^\varphi(Y_k,r)\right|^{2p}
 &= \mathbb{E}\Big|\int_r^t\sum_{x\in {\mathbb{X}}_{n_k}}a\left(Y_k(s,x)\right)
 \varphi(x)\sqrt{\delta_{n_k}}dW^x_{n_k}(s)\Big|^{2p}\\
 &\le K\mathbb{E} \Big|\sum_{x\in {\mathbb{X}}_{n_k}}\int_0^ta^2\left(Y_k(s,x)
 \right)\varphi^2(x)\delta_{n_k}ds\Big|^{p}
\end{split}
\label{ui2}
\end{equation}
for some constant $K>0$ independent of $k$.  Now, let ${\eta}_{_k}^t$ be the measure defined
on $[0,t]\times{\mathbb{X}}_{n_k}$ by $d{\eta}_{_k}^t(s,x)=\varphi^2(x)\delta_{n_k}ds$ and let
$|{\eta}_{_k}^t| = {\eta}_{_k}^t([0,t] \times{\mathbb{X}}_{n_k})$.  Clearly, for a fixed $\varphi$,
$$
\sup_{\substack{k\in\mathbb{N}\\0\le t\le T}}|{\eta}_{_k}^t|\le K
$$
for some constant $K>0$ independent of $k$
($K$ depends only on $T$, $\sup_x\varphi^2(x)$, and the Lebesgue measure of the support of $\varphi$).
Then, rewriting \eqref{ui2} and---observing that ${\eta}_{_k}^t/|{\eta}_{_k}^t|$
is a probability measure---applying Jensen's inequality yields
\begin{equation}
\begin{split}
&\mathbb{E}\left|S^\varphi(Y_k,t)-S^\varphi(Y_k,r)\right|^{2p}\\
&\le K\mathbb{E}
\Big|\int_{[0,t]\times{\mathbb{X}}_{n_k}} a^2\left(Y_k(s,x)\right)
\frac{d{\eta}_{_k}^t(s,x)}{|{\eta}_{_k}^t|}\Big|^{p}|{\eta}_{_k}^t|^p
\\ &\le {|{\eta}_{_k}^t|}^{p-1}K\int_{[0,t]\times{\mathbb{X}}_{n_k}}\mathbb{E}\left[a^{2p}\left(Y_k(s,x)\right)\right]d{\eta}_{_k}^t(s,x)\\
 &\le K\sum_{x\in {\mathbb{X}}_{n_k}}\int_0^T\mathbb{E}\left[a^{2p}\left(Y_k(s,x)\right)\right]\varphi^{2}(x){\delta_{n_k}}ds
 \le K <\infty,
\end{split}
\label{ui70}
\end{equation}
for some constant $K>0$ independent of $k$, where in the next to last inequality we also used the growth condition on $a$ ((b) in \eqref{acond})
and Proposition \ref{Expbd} ($Y_k \overset{d}{=} \Tilde{U}_{n_k}$),
along with the compact supportedness and boundedness of $\varphi$.   Thus,
\begin{equation}
\lim_{k\to\infty}\mathbb{E}\left[\left.\left(S^\varphi(Y_k,t)-S^\varphi(Y_k,r)\right)^2\right|\eus{G}_r\right]\\
=\mathbb{E}\left[\left.\left(M^\varphi(t)-M^\varphi(r)\right)^2\right|\eus{G}_r\right].
\label{expl1}
\end{equation}
Also, for the same reasons as in the next to last inequality in \eqref{ui70}, we see that
$$
\mathbb{E}\int_r^t\sum_{x\in {\mathbb{X}}_{n_k}}a^2\,\left(Y_k(s,x)\right)
\varphi^2(x)\delta_{n_k}ds\le K<\infty,
$$
for some constant $K>0$ independent of $k$.  Therefore, for each $r,t$ and
each $\varphi$,
$$
\Big\{\int_r^t\sum_{x\in {\mathbb{X}}_{n_k}}a^2\,\left(Y_k(s,x)\right)
\varphi^2(x)\delta_{n_k}ds\Big\}_k
$$
is a uniformly integrable sequence and thus
\begin{equation}
\begin{split}
&\lim_{k\to\infty}\mathbb{E}\Big[\int_r^t\sum_{x\in {\mathbb{X}}_{n_k}}a^2\,
\left(Y_k(s,x)\right)\varphi^2(x)\delta_{n_k}ds\Big|\eus{G}_r\Big]\\
&=\mathbb{E}\Big[\int_r^t\int_{\mathbb{R}}a^2\left(Y^x(s)\right)\varphi^2(x)dx ds
\Big|\eus{G}_r\Big].
\end{split}
\label{expl2}
\end{equation}
Now, equations  \eqref{expl}, \eqref{expl1}, and \eqref{expl2} yield
\begin{equation}
\left\langle{M^\varphi(\cdot)}\right\rangle_t = \int_0^t \int_{{\mathbb{R}}}a^2(Y^x(s))\varphi^2(x)dx ds,
\label{QVoftheY}
\end{equation}
and (ii) is proved.
\end{proof}

\subsection{Regularity and Uniqueness}

Having established existence for the SPDE
$e^{\epsilon_1,\epsilon_2}_{\mbox{\tiny heat}} (a,b,\xi)$ under our conditions
\eqref{acond}, we turn to the proof of some properties of our solution $Y$.

\begin{proof}[Proof of  the regularity part of Theorem \ref{ww}]
We divide the proof in two steps:

\noindent (1) $Y$ is $L^p$ bounded for all $p\ge2$:\quad
First, note that $Y_k\overset{d}{=}\tilde{U}_{n_k} $ and Proposition \ref{Expbd} give us,  for each
$q\ge1$:
 \begin{equation}
 \mathbb{E}\left|Y_k(t,x)\right|^{2q}
 =\mathbb{E}\left|\tilde{U}^x_{n_k}(t)\right|^{2q}
 \le K\exp{(KT)}<\infty;\quad  \forall (k,t,x)\in\mathbb{N}\times\mathbb{T}\times\mathbb{R},
 \label{ui7}
 \end{equation}
for some constant $K$  (independent of $k,t,x$).  It follows that, for each $(t,x)\in\mathbb{T}\times\mathbb{R}$ the sequence $\{|Y_k(t,x)|^p\}_k$ is uniformly
integrable for each $p\ge2$.  Thus,
\begin{equation}
\mathbb{E}|Y(t,x)|^{p}\le\lim_{k\to\infty}\mathbb{E}\left|Y_k(t,x)\right|^{p}
\le K<\infty;\ \forall (t,x)\in\mathbb{T}\times\mathbb{R},\ \forall p\ge2,
 \label{Lp}
\end{equation}
and the desired conclusion follows.

\noindent (2) The continuous paths of $Y$ are H\"older $\gamma_s\in(0,\frac12)$
in space and H\"older $\gamma_t\in(0,\frac14)$ in time:\quad
Using Proposition \ref{Expbd}, we get, for each $q\ge1$, that
\begin{equation}
\begin{split}
&\mathbb{E}\left|Y_k(t,x)-Y_k(t,y)\right|^{2q}+\mathbb{E}\left|Y_k(t,x)-Y_k(r,x)\right|^{2q}\\&\le K\left(\mathbb{E}|Y_k(t,x)|^{2q}+\mathbb{E}|Y_k(t,y)|^{2q}+\mathbb{E}|Y_k(r,x)|^{2q}\right)\\
&\le K;\quad \forall (k,r,t,x,y)\in\mathbb{N}\times\mathbb{T}^2\times\mathbb{R}^2.
\end{split}
\label{ui14}
\end{equation}
So, for each $(r,t,x,y)\in\mathbb{T}^2\times\mathbb{R}^2$, the sequences
$\big\{\left|Y_k(t,x)-Y_k(t,y)\right|^{2q}\big\}_k$  and
$\big\{\left|Y_k(t,x)-Y_k(r,x)\right|^{2q}\big\}_k$ are uniformly integrable,
for each $q\ge1$.  Therefore, using Lemma \ref{boundsonsdifferences}
and Lemma \ref{boundsontdifferences}, we obtain
 \begin{equation}
\begin{split}
(i)\quad&\mathbb{E}\left|Y(t,x)-Y(t,y)\right|^{2q}\\
&=\lim_{k\to\infty}\mathbb{E}\left|Y_k(t,x)-Y_k(t,y)\right|^{2q}\\
&=\lim_{k\to\infty}\mathbb{E}
\left|\tilde{U}_{n_k}^x(t)-\tilde{U}_{n_k}^y(t)\right|^{2q}
\le K|x-y|^q; \quad \mbox{whenever } |x-y|<1,\\
(ii)\quad&\mathbb{E}\left|Y(t,x)-Y(r,x)\right|^{2q}\\
&=\lim_{k\to\infty}\mathbb{E}\left|Y_k(t,x)-Y_k(r,x)\right|^{2q}\\
&=\lim_{k\to\infty}\mathbb{E}\left|\tilde{U}_{n_k}^x(t)-\tilde{U}_{n_k}^x(r)
\right|^{2q} \le K|t-r|^{q/2}; \quad\mbox{whenever } |t-r|<1.
\end{split}
\label{timespace}
\end{equation}
Now, letting $q_n=n+1$ for $n\in\{0,1,\ldots\}$ and let $n=m+1$ for
$m\in\{0,1,\ldots\}$, we then have from  \eqref{timespace} that
 \begin{equation}
\begin{split}
(i)\quad  &\mathbb{E}\left|Y(t,x)-Y(t,y)\right|^{2+2n}
\le K|x-y|^{1+n};\quad\mbox{whenever } |x-y|<1,\\
(ii)\quad &\mathbb{E}\left|Y(t,x)-Y(r,x)\right|^{4+2m}
\le K|t-r|^{1+\frac{m}{2}}; \quad\mbox{whenever } |t-r|<1.
\end{split}
\label{timespace2}
\end{equation}
By Theorem 2.8 p.~53 \cite{KS} we get that  $\gamma_s\in(0,\frac{n}{2n+2}) $ and
$\gamma_t\in(0,\frac{m/2}{2m+4})$ $ \forall m,n$, from which the
proof follows upon taking the limits as $m,n\to\infty$.
\end{proof}

\begin{proof}[Proof of the uniqueness part of Theorem \ref{ww}]
Consider $e^{\epsilon_1,\epsilon_2}_{\mbox{\tiny heat}} (a,b,\xi)$
on the rectangle $\eus{R}_{T,L}\overset{\triangle}{=} [0,T]\times[0,L]$
for some $T,L>0$, and assume that $a(u)$
and $b(u)$ are as given in Theorem \ref{ww}.
Then as in the proof of Theorem 1.2 in \cite{ACR1} (see also the comment after Remark 1.1 in \cite{ACR2}), we only need to show that, if $\lambda$ is
Lebesgue measure on $\eus{R}_{T,L}$, then the ratios
$b(U)/a(U)$ and $b(V)/a(V)$ are in $L^2(\eus{R}_{T,L},\lambda)$
almost surely whenever $U$ solves
$e^{\epsilon_1,\epsilon_2}_{\mbox{\tiny heat}}(a,0,\xi)$ and $V$ solves
$e^{\epsilon_1,\epsilon_2}_{\mbox{\tiny heat}}(a,b,\xi)$.    But this easily
follows as in the proof of Theorem 1.2 \cite{ACR1} under our conditions, since we always assume that solutions to
$e^{\epsilon_1,\epsilon_2}_{\mbox{\tiny heat}} (a,b,\xi)$, and hence $U$ and $V$ are continuous.
\end{proof}

\section{The vanishing of the Laplacian vs.~noise}

We now prove of Theorem \ref{thm2}, which asserts that
${\epsilon_2}/{{ \epsilon_1}^{1/4}}$ is the correct scaling of
$\epsilon_1$ and $\epsilon_2$ when we
investigate the asymptotic behavior as $\epsilon_1,\epsilon_2\to0$.

\begin{proof}[Proof of Theorem \ref{thm2}]
Throughout this proof we use the process $Y$ of Theorem \ref{ww},
Remark \ref{fixedepsilon}, and
Remark \ref{rm2} to get to the desired conclusions.
This is justified by the fact that $Y$ has the same law as $U$.

\noindent (i)\quad
We prove it by contradiction.  So, assume there is a $T>0$ such that
$$
\lim_{\substack{\epsilon_1,\epsilon_2\downarrow0\\
{\epsilon_2}/{{ \epsilon_1}^{1/4}}\to\infty}}
\sup_{0\le s\le T}\sup_{{x\in\mathbb{R}}}\mathbb{E} Y_{\epsilon_1,\epsilon_2}^{2}(s,x)
<\infty
$$
and assume without loss of generality that $\xi\equiv0$.  Observe that
\begin{equation}
\begin{split}
&\mathbb{E}\left|Y_{\epsilon_1,\epsilon_2}(t,x)\right|^2\\
&=\mathbb{E} \Big| \int_{\mathbb{R}}\int_0^t G_{\epsilon_1}(s,t;x,y)\left[{\epsilon_2a(Y_{\epsilon_1,\epsilon_2}(s,y))} \eus{W}(ds,dy)+
b(Y_{\epsilon_1,\epsilon_2}(s,y))ds\,dy\right]\Big|^{2}\\
&= \epsilon_2^2\int_{\mathbb{R}}\int_0^t G_{\epsilon_1}^2(s,t;x,y)
{\mathbb{E} a^2(Y_{\epsilon_1,\epsilon_2}(s,y))} ds\,dy \\
&\quad + \mathbb{E}\Big(\int_{\mathbb{R}}\int_0^t G_{\epsilon_1}(s,t;x,y)
b(Y_{\epsilon_1,\epsilon_2}(s,y))dsdy\Big)^2\\
&\quad + 2\epsilon_2\mathbb{E}\Big(\int_{\mathbb{R}}\int_0^t G_{\epsilon_1}
(s,t;x,y)a(Y_{\epsilon_1,\epsilon_2}(s,y)) \eus{W}(ds,dy)\\
&\quad \times\int_{\mathbb{R}}\int_0^t G_{\epsilon_1}(s,t;x,y)
b(Y_{\epsilon_1,\epsilon_2}(s,y))dsdy\Big)\\
&\ge K^2_{l}\epsilon_2^2\int_{\mathbb{R}}\int_0^t G_{\epsilon_1}^2(s,t;x,y) ds\,dy
+\mathbb{E}\Big(\int_{\mathbb{R}}\int_0^t G_{\epsilon_1}(s,t;x,y)
b(Y_{\epsilon_1,\epsilon_2}(s,y))dsdy\Big)^2\\
&\quad  2\epsilon_2\mathbb{E}\Big(\int_{\mathbb{R}}\int_0^t G_{\epsilon_1}
(s,t;x,y)a(Y_{\epsilon_1,\epsilon_2}(s,y)) \eus{W}(ds,dy)\\
&\quad\times\int_{\mathbb{R}}\int_0^t G_{\epsilon_1}(s,t;x,y)
b(Y_{\epsilon_1,\epsilon_2}(s,y))dsdy\Big),
\end{split}
\label{as00}
\end{equation}
where we used the assumption $0<K_{l}\le a(u)$ to get the last inequality in \eqref{as00}.
Now, denoting by $P_L$ the product inside the expectation in the last term in \eqref{as00}, applying Cauchy-Schwarz, using the assumption
that $a(u)\le K_{L}$, the fact
\begin{equation}
 \int_{\mathbb{R}}\int_0^t {G_{\epsilon_1}^2(s,t;x,y)}dsdy={\frac {\sqrt { t }}{\sqrt {\pi \epsilon_1}}},
\label{L2green}
\end{equation}
and letting $C_T= 2K_L T^{1/4}/{\pi}^{1/4}$, we get that
\begin{equation}
\begin{split}
\left|2\epsilon_2\mathbb{E} P_L\right|
&\le2\epsilon_2\sqrt{\int_{\mathbb{R}}\int_0^t G_{\epsilon_1}^2(s,t;x,y)\mathbb{E}
a^2(Y_{\epsilon_1,\epsilon_2}(s,y)) dsdy}\\
&\quad \times\sqrt{\mathbb{E}\Big(\int_{\mathbb{R}}\int_0^t {G_{\epsilon_1}
(s,t;x,y))}b(Y_{\epsilon_1,\epsilon_2}(s,y))dsdy\Big)^2}\\
&\le C_T{\frac {\epsilon_2}{{ \epsilon_1}^{1/4}}}
\Big[\mathbb{E}\Big(\int_{\mathbb{R}}\int_0^t {G_{\epsilon_1}(s,t;x,y))}
b(Y_{\epsilon_1,\epsilon_2}(s,y))dsdy\Big)^2\Big]^{1/2}
\end{split} \label{CS}
\end{equation}
Now, for a fixed point
$(t,x,\epsilon_1)\in\mathbb{T}\times\mathbb{R}\times\mathbb{R}_+$ let
$\nu^{t,x}_{\epsilon_1}$ be the measure on $[0,t] \times\mathbb{R}$ defined by
$d\nu^{t,x}_{\epsilon_1}(s,y)=G_{\epsilon_1}(s,t;x,y)ds dy$, and let
$|\nu^{t,x}_{\epsilon_1}| = \nu^{t,x}_{\epsilon_1}([0,t] \times \mathbb{R})$.
Then, observing that
\begin{equation}
\left|\nu^{t,x}_{\epsilon_1}\right|
=\Big(\int_{\mathbb{R}}\int_0^t {G_{\epsilon_1}(s,t;x,y)}dsdy\Big)=t,
\label{Lgreen}
\end{equation}
and that $\nu^{t,x}_{\epsilon_1}/|\nu^{t,x}_{\epsilon_1}|$ is a
probability measure, we apply Jensen's inequality and the growth condition on
$b$ to \eqref{CS} to get
 \begin{equation}
\begin{split}
\left|2\epsilon_2\mathbb{E} P_L\right|
&\le C_T{\frac {\epsilon_2}{{ \epsilon_1}^{1/4}}}
\Big(\int_{[0,t] \times \mathbb{R}}\mathbb{E} b^2(Y_{\epsilon_1,\epsilon_2}
(s,y))\frac{d\nu^{t,x}_{\epsilon_1}(s,y)}
{\left|\nu^{t,x}_{\epsilon_1}\right|}\Big)^{1/2}
{\left|\nu^{t,x}_{\epsilon_1}\right|}\\
&\le \tilde{C}_T{\frac {\epsilon_2}{{ \epsilon_1}^{1/4}}}
\Big(\int_{[0,t] \times \mathbb{R}}\mathbb{E} b^2(Y_{\epsilon_1,\epsilon_2}(s,y))\frac{d\nu^{t,x}_{\epsilon_1}(s,y)}
{\left|\nu^{t,x}_{\epsilon_1}\right|}\Big)^{1/2}\\
&\le\tilde{C}_T{\frac {\epsilon_2}{{ \epsilon_1}^{1/4}}}\sup_{0\le s\le T}
\sup_{{x\in\mathbb{R}}}\left(\mathbb{E} b^2(Y_{\epsilon_1,\epsilon_2}(s,x))\right)^{1/2}
\\
&\le K_T{\frac {\epsilon_2}{{ \epsilon_1}^{1/4}}}\sup_{0\le s\le T}
\sup_{{x\in\mathbb{R} }}\left(1+\mathbb{E} Y_{\epsilon_1,\epsilon_2}^{2}(s,x)\right)^{1/2}
\\
&\le K_T{\frac {\epsilon_2}{{ \epsilon_1}^{1/4}}}
\Big(1+\sup_{0\le s\le T}\sup_{{x\in\mathbb{R}}}\left[\mathbb{E} Y_{\epsilon_1,
\epsilon_2}^{2}(s,x)\right]^{1/2}\Big).
\end{split} \label{Jens}
\end{equation}
Equations \eqref{Jens}, \eqref{L2green}, and \eqref{as00} then yield
\begin{equation}
\mathbb{E}\left|Y_{\epsilon_1,\epsilon_2}(t,x)\right|^2-K^2_{l}
\sqrt{\frac{ t }{\pi}}\frac{\epsilon_2^2}{\sqrt{\epsilon_1}}
+K_T{\frac {\epsilon_2}{{ \epsilon_1}^{1/4}}}\Big(1+\Big[\sup_{0\le s\le T}
\sup_{{x\in\mathbb{R}}}\mathbb{E}
Y_{\epsilon_1,\epsilon_2}^{2}(s,x)\Big]^{1/2}\Big)\ge 0.
\label{contr}
\end{equation}
Taking the limit as $\epsilon_1,\epsilon_2\searrow0$ in \eqref{contr} such that ${\epsilon_2}/{{ \epsilon_1}^{1/4}}\to\infty$
and using the finiteness assumption
on $\sup_{0\le s\le T}\sup_{{x\in\mathbb{R}}}\mathbb{E} Y_{\epsilon_1,\epsilon_2}^{2}(s,x)$ (and hence the finiteness of $\mathbb{E}\left|Y_{\epsilon_1,\epsilon_2}(t,x)\right|^2$),
we obtain the desired
contradiction (since the negative term is of order ${\epsilon_2^2}/{\sqrt{\epsilon_1}}$).     The fact that $Y_{\epsilon_1,\epsilon_2}$ has the same law as our
SDDEs limit solution $U_{\epsilon_1,\epsilon_2}$ completes the proof.  The proof for the case or $\epsilon_1,\epsilon_2\nearrow\infty$
follows exactly the same steps.

\noindent(ii)\quad
  The difference between our SPDE and its deterministic counterpart, whose solution we denote by $U_{\epsilon_1}$, is given by
\begin{equation}
\begin{split}
Y_{\epsilon_1,\epsilon_2}(t,x)-U_{\epsilon_1}(t,x)
&=\int_{\mathbb{R}}\int_0^t G_{\epsilon_1}(s,t;x,y){\epsilon_2
a(Y_{\epsilon_1,\epsilon_2}(s,y))} \eus{W}(ds,dy)\\
&\quad+ \int_{\mathbb{R}}\int_0^t G_{\epsilon_1}(s,t;x,y)\Big\lbrack
b(Y_{\epsilon_1,\epsilon_2}(s,y))-b(U_{\epsilon_1})\Big\rbrack ds\,dy.
\end{split}
\label{sandddiff}
\end{equation}
Let $\nu^{t,x}_{\epsilon_1}$ be the measure on $[0,t] \times\mathbb{R}$
defined as in part (i) above, and let $d\mu^{t,x}_{\epsilon_1}(s,y)
=G^2_{\epsilon_1}(s,t;x,y)ds\,dy$  and let
$|\mu^{t,x}_{\epsilon_1}| = \mu^{t,x}_{\epsilon_1}([0,t] \times \mathbb{R})$;
then, using Burkholder and Jensen's inequalities;
and finally using the boundedness on $a$ and the Lipschitz continuity
assumption on $b$ and the simple fact that
$\int_{\mathbb{R}}G_{\epsilon_1}(s,t;x,y)dy=1$,
we  get for $0\le t\le T$ that
\begin{equation}
\begin{split}
&\mathbb{E}\left|Y_{\epsilon_1,\epsilon_2}(t,x)-U_{\epsilon_1}(t,x)\right|^{2q}\\
&\le K\Big\{\mathbb{E}\Big|\int_{\mathbb{R}}\int_0^t G_{\epsilon_1}(s,t;x,y)
{\epsilon_2a(Y_{\epsilon_1,\epsilon_2}(s,y))} \eus{W}(ds,dy)\Big|^{2q}\\
&\quad+ \mathbb{E}\Big|\int_{\mathbb{R}}\int_0^t G_{\epsilon_1}(s,t;x,y)
\left[b(Y_{\epsilon_1,\epsilon_2}(s,y))-b(U_{\epsilon_1}(s,y))\right]ds,dy
\Big|^{2q}\Big\}\\
&\le K\Big\{\epsilon_2^{2q}\int_{\mathbb{R}}\int_0^t\mathbb{E}
a^{2q}(Y_{\epsilon_1,\epsilon_2}(s,y))
\frac{d\mu^{t,x}_{\epsilon_1}(s,y)}{|\mu^{t,x}_{\epsilon_1}|}|
\mu^{t,x}_{\epsilon_1}|^{q}\Big\}\\
&\quad +K\Big\{\int_{\mathbb{R}}\int_0^t\mathbb{E}
\left|b(Y_{\epsilon_1,\epsilon_2}(s,y))-b(U_{\epsilon_1})\right|^{2q}
\frac{d\nu^{t,x}_{\epsilon_1}(s,y)}{|\nu^{t,x}_{\epsilon_1}|}|
\nu^{t,x}_{\epsilon_1}|^{2q}\Big\} \\
&\le Kt^{2q-1}\Big\{\int_{\mathbb{R}}\int_0^t\mathbb{E} \left|Y_{\epsilon_1,\epsilon_2}
(s,y)-U_{\epsilon_1}(s,y))\right|^{2q}d\nu^{t,x}_{\epsilon_1}(s,y)\Big\}
+\frac{Kt^{q/2}\epsilon_2^{2q}}{(\pi \epsilon_1)^{q/2}}\\
&\le KT^{2q-1}\Big\{\int_0^t\sup_{ x\in\mathbb{R}}\mathbb{E}
\left|Y_{\epsilon_1,\epsilon_2}(s,x)-U_{\epsilon_1}(s,x))\right|^{2q}ds\Big\}
+\frac{KT^{q/2}\epsilon_2^{2q}}{(\pi \epsilon_1)^{q/2}}
\end{split}
\label{momineq1}
\end{equation}
Letting $C_T=K(1\vee T^{2q-1})$ we get, upon applying Gronwall's Lemma,
$$
\sup_{ 0\le t\le T}\sup_{x\in\mathbb{R} }\mathbb{E}\left|Y_{\epsilon_1,\epsilon_2}
(t,x)-U_{\epsilon_1}(t,x)\right|^{2q}
\le \frac{C_T\epsilon_2^{2q}}{(\pi \epsilon_1)^{q/2}}e^{C_TT}\to 0
$$
as $\epsilon_1,\epsilon_2$, and
$\epsilon_2/\epsilon_1^{1/4}$ approach $0$.  The conclusion follows from the fact that $Y_{\epsilon_1,\epsilon_2}$
has the same law as $U_{\epsilon_1,\epsilon_2}$.
Denoting the solution to the deterministic counterpart of our approximating (discrete space) SDDE by $\tilde{U}_{n,\epsilon_1}$, replacing
the integral over space ($\mathbb{R}$) with sum over the lattice ($\mathbb{X}_n$) in the above argument, replacing the Green function $G_{\epsilon_1}$ above
with the random walk density $Q_{\delta_n,\epsilon_1}$, and following the same steps as above we get
$$
\sup_{ 0\le t\le T}\sup_{x\in\mathbb{X}_n }\mathbb{E}\left|\tilde{U}^x_{n,\epsilon_1,
\epsilon_2}(t)-\tilde{U}_{n,\epsilon_1}^x(t)\right|^{2q}\to 0\quad
\mbox{as }\epsilon_1,\epsilon_2, \epsilon_2/\epsilon_1^{1/4}\to 0.
$$
\end{proof}

\begin{rem} \rm
In contrast to part (ii) of the above proof, the argument in part (i)
doesn't work for the approximating SDDEs, this becomes clear upon comparing
Lemmas \ref{2ndQinequality} and \eqref{L2green}.
\label{SDDEno}
\end{rem}

\subsection*{Acknowledgement}
I would like to thank Robin Pemantle for pointing out to me the elegant
coupling proof of Lemma \ref{5thinequality}.
I would also like to thank the referee for his positive and thorough comments.

\section{Appendix}
We now show that the (local)-martingale problem in Theorem \ref{1stmartprob}
is equivalent to the test function
formulation of $e^{\epsilon_1,\epsilon_2}_{\mbox{\tiny heat}} (a,b,\xi)$.
To simplify notations, we assume $\epsilon_1=\epsilon_2=1$
(the case of general parameters is proven in
the same way with only obvious notational differences).
This equivalence holds as well for the $\mathbb{R}^d$, $d>1$, case;
and we will prove it in this generality.

\begin{thm}
Assume that $a,b,$ and $\xi$ satisfy the conditions in $\eqref{acond}$.  Then,
the (local) martingale problem in Theorem \ref{1stmartprob} is equivalent
to the test function formulation of
$e^{\epsilon_1,\epsilon_2}_{\mbox{\tiny heat}} (a,b,\xi)$.
\label{marttff}
\end{thm}

\begin{proof}
If the test function formulation of $e^{\epsilon_1,\epsilon_2}_{\mbox{\tiny heat}} (a,b,\xi)$ holds on
$(\Omega,\eus{F}\{\eus{F}_t\},\mathbb{P})$, then
\begin{equation}
\begin{split}
M^\varphi(t)&\overset{\triangle}{=} (U(t)-\xi,\varphi)-\frac12 \int_0^t(U(s),
\varphi'')ds-\int_0^t(b(U(s)),\varphi)ds\\
&=\int_0^t \int_{\mathbb{R}^d} a(U(s,x))\varphi(x)\eus{W}(dx,ds)\\
&=\int_0^t \int_{{{\mathbb S}}^\varphi}a(U(s,x)) \varphi(x)\eus{W}(dx,ds),
\end{split} \label{TFFa}
\end{equation}
where ${{\mathbb S}}^\varphi\subset\mathbb{R}^d$ is the compact support of
$\varphi$.  It follows from the
assumptions on $a$ and the boundedness of $\varphi$ that $M^\varphi(t)$ is an
$\eus{F}_t$-local martingale under $\mathbb{P}$ and that
\begin{equation}
\langle M^\varphi(\cdot)\rangle_t=\int_0^t
\int_{S^\varphi}a^2(U(s,x))\varphi^2(x)dx\,ds=
\int_0^t \int_{\mathbb{R}^d}a^2(U(s,x))\varphi^2(x)dxds.
\label{QVa}
\end{equation}
For the other direction, assume that $M^\varphi(t)$, as defined in \eqref{TFFa},
is a local martingale on
$(\Omega,\eus{F}\{\eus{F}_t\},\mathbb{P})$ with
quadratic variation given by \eqref{QVa}.  Suppose also that $a$ vanishes almost
nowhere in $(u,\omega)\in\mathbb{R}\times\Omega$ (if this fails we can always
do the same
as in the finite dimensional case cf.~Ikeda and Watanabe \cite{IW} or
Doob \cite{Db}).  Now let $\lambda$ denote Lebesgue measure on
$\eus{B}({\mathbb{R}_+\times\mathbb{R}^d})$ and on $\eus{B}({\mathbb{R}^d})$;
and, for each $t\ge 0$, define the random measure
$\eus{M}_t(A)=\eus{M}([0,t]\times A)$ on the ring
$\eus{R}\overset{\triangle}{=} \{A\in\eus{B}(\mathbb{R}^d);\lambda([0,t]
\times A)<\infty,\ \forall t>0\}$
by the recipe
\begin{equation}
\begin{gathered}
\int_0^t\int_{A}\varphi(x)\eus{M}(dx,ds)
=\int_0^t\int_{\mathbb{R}^d}\varphi(x)\eus{M}(dx,ds)
\overset{\triangle}{=}  M^\varphi(t); \\ \forall A\in\eus{R}\mbox{ with }
\lambda(A\triangle{{\mathbb S}}^\varphi)=0 \mbox{ or such that } A
\supset{{\mathbb S}}^\varphi,\ \forall\varphi\in C_c^\infty(\mathbb{R}^d;\mathbb{R}).
\end{gathered}
\label{martdef}
\end{equation}
By assumption, we have that $M^\varphi(t)$ is a continuous local martingale
for each $\varphi\in C_c^\infty(\mathbb{R}^d;\mathbb{R})$.
Furthermore, if $\varphi_1,\varphi_2\in C_c^\infty(\mathbb{R}^d;\mathbb{R})$
have disjoint supports
(${{\mathbb S}}^{\varphi_1}\cap{{\mathbb S}}^{\varphi_2}=\phi$),
then for any disjoint $A,B\in\eus{R}$ with
$\lambda(A\triangle{{\mathbb S}}^\varphi_1)=0$ and
$\lambda(B\triangle{{\mathbb S}}^\varphi_2)=0$ we
have by the definition of $M^\varphi(t)$, the fact that the second and third
terms in $M^\varphi(t)$'s definition in \eqref{TFFa} are of bounded variation,
\eqref{QVa}, and \eqref{martdef}
\begin{equation}
\begin{split}
&\Big\langle\int_0^{\cdot}\int_{A}\varphi_1(x)\eus{M}(dx,ds),
\int_0^{\cdot}\int_{B}\varphi_2(x)\eus{M}(dx,ds)\Big\rangle_t\\
&=\langle(U(t)-\xi,\varphi_1),(U(t)-\xi,\varphi_2)\rangle_t\\
&=\frac14\left[\langle(U(t)-\xi,\varphi_1+\varphi_2)\rangle_t -
\langle(U(t)-\xi,\varphi_1-\varphi_2)\rangle_t\right]\\
&=\int_0^t\int_{\mathbb{R}^d}a^2(U(s,x)) \varphi_1(x)\varphi_2(x)dx\,ds=0\,.
\end{split} \label{ortho}
\end{equation}
Thus, $\eus{M}$ is a continuous orthogonal local martingale measure \cite{A298}.  By the
quadratic variation assumption on $M^\varphi(t)$, we also have that for each
$\varphi\in C_c^\infty(\mathbb{R}^d;\mathbb{R})$
\begin{equation}
\begin{split}
\int_0^t \int_{\mathbb{R}^d}\varphi^2(x)a^2(U(s,x))dxds&=
\Big\langle\int_0^{\cdot}\int_{\mathbb{R}^d}\varphi(x)\eus{M}(dx,ds)\Big\rangle_t \\
&=\int_0^{t}\int_{\mathbb{R}^d}\varphi^2(x)\nu_{\eus{M}}(dx,ds).
\end{split}
\label{QVB}
\end{equation}
So that the intensity measure $\nu_{\eus{M}}$ of $\eus{M}$ is given by
\begin{equation}
\nu_{\eus{M}}(dx,ds)=a^2(U(s,x))dxds,\mbox{ on sets of the form } [0,t]\times A,\
A\in\eus{R}.
\label{intensity}
\end{equation}
We now show that there is a space-time white noise $\eus{W}$ such that
\begin{equation}
\int_0^t\int_{\mathbb{R}^d}\varphi(x)\eus{M}(dx,ds)
=\int_0^t \int_{\mathbb{R}^d}a(U(s,x))\varphi(x)\eus{W}(dx,ds),\
\forall\varphi\in C_c^\infty(\mathbb{R}^d;\mathbb{R}).
\label{Kalas}
\end{equation}
For each $A\in\eus{R}$, let
$$
\eus{W}_t(A)\overset{\triangle}{=}
\int_0^t\int_{A}\frac{\eus{M}(dx,ds)}{a(U(s,x))}.
$$
$\eus{W}=\{\eus{W}_t(A);t\in\mathbb{R}_+, A\in\eus{R}\}$ is clearly a
continuous orthogonal local martingale measure with intensity
$\nu_{\eus{W}}=\lambda([0,t]\times A),$ where $\lambda$ is Lebesgue measure,
 so it is a white noise and clearly
\eqref{Kalas} holds, completing the proof.
\end{proof}

\end{document}